\documentclass{amsart}
\usepackage{graphicx}
\usepackage{amscd}
\usepackage{amsmath}
\usepackage{amsthm}
\usepackage{amsfonts}
\usepackage{amssymb}
\usepackage{mathrsfs}
\usepackage{enumerate}
\usepackage{amsrefs}
\usepackage[ pdfstartview=FitB]{hyperref}

\newcommand{\D}{\operatorname{\mathbb{D}}}
\newcommand{\N}{\operatorname{\mathbb{N}}}

\newcommand{\C}{\operatorname{\mathbb{C}}}

\newcommand{\e}{\operatorname{\varepsilon}}

\newcommand{\ol}{\overline }
\newcommand{\hil}{\operatorname{\mathcal{H}}}
\newcommand{\kil}{\operatorname{\mathcal{K}}}
\let\phi\varphi

\newtheorem{lemma}{Lemma}[section]
\newtheorem{theorem}[lemma]{Theorem}
\newtheorem{proposition}[lemma]{Proposition}
\newtheorem{corollary}[lemma]{Corollary}
\theoremstyle{definition}

\newtheorem*{question}{Question}

\begin{document}
\author{Rapha\"el Clou\^atre}
\address{Department of Mathematics, Indiana University, 831 East 3rd Street,
Bloomington, IN 47405}
\email{rclouatr@indiana.edu}
\title[Similarity results for operators of class $C_0$]{Similarity results for operators of class $C_0$ and the algebra $H^\infty(T)$}
\begin{abstract}
Given two multiplicity-free operators $T_1$ and $T_2$ of class $C_0$ having the same finite Blaschke product as minimal function, the operator algebras $H^\infty(T_1)$ and $H^\infty (T_2)$ are isomorphic and $T_1$ is similar to $T_2$. We find conditions under which the norm of the similarity between the operators can be controlled by the norm of the algebra isomorphism. As an application, we improve upon earlier work  and obtain results regarding similarity when the minimal function is an infinite product of finite Blaschke products satisfying the generalized Carleson condition.
\end{abstract}

\maketitle

\section{Introduction}

One of the main features of $C_0$ contractions  is their classification up to quasisimilarity in terms of the Jordan models. It is a natural impulse to wonder whether something could be said about \textit{similarity} classes, and this is our aim here. In fact, our concern in this paper is that of determining conditions under which we can improve the quasisimilarity of an operator $T$ of class $C_0$ with its Jordan model to similarity. Early results include those of \cite{apostol}, which inspired the work done in \cite{clouatre}. The corresponding question for unitary equivalence was investigated by Arveson in his seminal paper \cite{arveson}. More recently, there has been some interest in this type of question in the setting of truncated Toeplitz operators (\cite{cima}, \cite{garcia}). 

For the class $C_0$, the problem was considered in \cite{clouatre} where motivation was provided and some partial results were obtained. The point of view we would like to adopt here is different from that of \cite{clouatre} in the sense that the basic assumption will be that the algebras $H^\infty(T)$ and $H^\infty (S(\theta))$ are boundedly isomorphic, instead of $\phi(T)$ having closed range for every inner divisor $\phi$ of $\theta$ (here $\theta$ denotes the minimal function of $T$). Theorem \ref{t-equivalencecarleson} below relates those two settings.

\begin{question}
Let $T_1\in B(\hil_1), T_2\in B(\hil_2)$ be multiplicity-free operators of class $C_0$ with the property that the algebras $H^\infty(T_1)$ and $H^\infty(T_2)$ are boundedly isomorphic. Does it follow that $T_1$ and $T_2$ are similar?
\end{question}

We obtain a quantitative answer to this question in the case where the minimal function is a finite Blaschke product. Let us briefly describe our main result.
Let $T\in B(\hil)$ be a multiplicity-free operator of class $C_0$ with minimal function 
$$
\theta=b_{\lambda_1}\ldots b_{\lambda_N}.
$$
Assume that there exists a bounded algebra isomorphism
$$
\Psi:H^\infty(T)\to H^\infty(S(\theta))
$$
such that $\Psi(u(T))=u(S(\theta))$ for every $u\in H^\infty$ and such that $\|\Psi\|$ is close to $1$ (this is made precise in the actual statement, see Corollary \ref{c-simalg}). Then, there exists an invertible operator 
$$
X:\hil\to H(\theta)
$$
such that $XT=S(\theta) X$ and 
$$
\max\{\|X\|,\|X^{-1}\| \}\leq C(\Psi,N)
$$
where $C(\Psi,N)>0$ is a constant depending only on $\Psi$ and $N$. This norm control on the similarity $X$ was previously unknown, and it is our main contribution.

The plan of the paper is as follows. Section 2 deals with preliminaries. Section 3 contains the precise statement of the main result. It also deals with the case where the underlying Hilbert space has dimension two. This case turns out to be particularly nice since some assumptions can be dropped. In Section 4 we recall a concept from interpolation theory and explain how it applies to our purposes. Finally, we apply our main theorem in Section 5 to obtain a similarity result for operators of class $C_0$, extending work done in \cite{clouatre}.

\section{Preliminaries}

We give here some background concerning operators of class $C_0$.
Let $H^\infty$ be the algebra of bounded holomorphic functions on
the open unit disc $\D$. Let $\hil$ be a Hilbert space and $T$ a
bounded linear operator on $\hil$, which we indicate by
$T\in B(\hil)$. If $T\in B(\hil)$ is a completely non-unitary contraction, then its associated Sz.-Nagy--Foias $H^\infty$ functional calculus is an algebra homomorphism $\Phi: H^\infty \to
B(\hil)$ with the following properties:
\begin{enumerate}[(i)]
    \item $\|\Phi(u)\|\leq u$ for every $u\in H^\infty$
    \item $\Phi(p)=p(T)$ for every polynomial $p$
    \item $\Phi$ is continuous when $H^\infty$ and $B(\hil)$ are equipped with their respective weak-star topologies.
\end{enumerate}
We use the notation $\Phi(u)=u(T)$ for $u\in H^\infty$.
The contraction $T$ belongs to the class $C_0$ whenever $\Phi$ has a non-trivial kernel.
It is known in that case that $\ker \Phi=\theta H^\infty$ for some inner function $\theta$ called the minimal
function of $T$, which is uniquely determined up to
a scalar factor of absolute value one. We now give the first elementary result we will use. Recall that a function $u\in H^\infty$ divides another function $v\in H^\infty$ if $v=uf$ for some $f\in H^\infty$. Moreover, given $E\subset \hil$, we denote by $\bigvee E$ the smallest closed subspace containing $E$. 

\begin{lemma}[\cite{bercOTA} Theorem 2.4.6]\label{l-lcmspan}
Let $T\in B(\hil)$ be an operator of class $C_0$ with minimal function $\theta$. Given a family $\{\theta_n\}_n$ of inner divisors of $\theta$ whose least common inner multiple is $\theta$, we have
$$
\hil=\bigvee_{n}\ker \theta_n(T).
$$
\end{lemma}

We denote by $H^2$ the Hilbert space of functions
$$f(z)=\sum_{n=0}^\infty a_n z^n$$
holomorphic on the open unit disc, equipped with the norm
$$
\|f\|_{H^2}^2=\sum_{n=0}^\infty |a_n|^2.
$$
For any inner function $\theta\in H^\infty$, the space $H(\theta)=H^2\ominus \theta H^2$ is closed and invariant for $S^*$, the adjoint of the shift operator $S$ on $H^2$. The operator $S(\theta)$ defined by $S(\theta)^*=S^*|(H^2\ominus \theta H^2)$ is called a Jordan block; it is of class $C_0$ with minimal function $\theta$. Given $u\in H^\infty$, we have that
$$
\|u(S(\theta))\|=\|u\|_{H^\infty/\theta H^\infty}=\inf\{\|u+\theta f\|_{H^\infty}:f\in H^\infty \}.
$$
The following is another useful property of Jordan blocks.

\begin{lemma}[\cite{bercOTA} Proposition 3.1.10]\label{l-unitequiv}
Let $\phi$ be an inner divisor of the inner function $\theta$. Then, the operator $S(\theta)|\ker \phi(S(\theta))$ is unitarily equivalent to $S(\phi)$.
\end{lemma}

A vector $x\in \hil$ is said to be cyclic for $T\in B(\hil)$ if 
$$
\bigvee \{T^n x: n\geq 0\}=\hil.
$$ 
An operator having a cyclic vector is said to be multiplicity-free. 

\begin{theorem}[\cite{bercOTA} Theorem 2.3.6]\label{maxvector}
Let $T\in B(\hil)$ be a multiplicity-free operator of class $C_0$. Then, the set of
cyclic vectors for $T$ is a dense $G_\delta$ in $\hil$. 
\end{theorem}

A bounded linear operator $X:\hil\to \hil'$ is called a quasiaffinity if it is injective and has dense range. The following result is the classification theorem mentioned in the introduction. Its conclusion is summarized by saying that $T$ is \textit{quasisimilar} to $S(\theta)$. Note that it is not stated here in its full generality, however this simpler version will suffice since we will only deal with multiplicity-free operators.

\begin{theorem}[\cite{bercOTA} Theorem 3.2.3]\label{classifC0}
Let $T\in B(\hil)$ be a multiplicity-free operator of class $C_0$ with minimal function $\theta$. Then, there exist quasiaffinities $X: \hil\to H(\theta)$
and $Y:H(\theta)\to \hil$ with the property that $XT=S(\theta)X$ and $TY=YS(\theta).$
\end{theorem}

More details about all of the above background material can be found in \cite{bercOTA}.
Let us close this section by setting some notation that will be used throughout the paper. For $\lambda\in \D$ we set
$$
b_{\lambda}(z)=\frac{z-\lambda}{1-\overline{\lambda}z}
$$
and we denote by 
$$
\kappa_{\lambda}(z)=\frac{1}{1-\ol{\lambda}z}
$$
the reproducing kernel for $H^2$ at $\lambda\in \D$. Also set 
$$
e_{\lambda}=\kappa_{\lambda}/\|\kappa_{\lambda}\|_{H^2}^2.
$$
If $\theta$ is a Blaschke product vanishing at $\lambda$, then it is easily verified that 
$$
e_{\lambda}=P_{H(b_{\lambda})}1\in H(\theta).
$$

\section{Isomorphisms of the algebra $H^\infty(T)$}

Let us start by recording a few elementary computational facts.
\begin{lemma}\label{lemmacompkernel}
Let $\theta=b_{\lambda_1}\ldots b_{\lambda_N}$ and set $\psi_j=\theta/b_{\lambda_j}$ for each $1\leq j \leq N$. Then, for every $1\leq j,k\leq N$ we have
\begin{enumerate}
\item[\rm{(i)}] $$\langle b_{\lambda_j},b_{\lambda_k}\rangle_{H^2}
=1+\ol{\lambda_j}\frac{\lambda_k-\lambda_j}{1-\ol{\lambda_j}\lambda_k}+\ol{\lambda_k}\frac{\lambda_j-\lambda_k}{1-\ol{\lambda_j}\lambda_k}$$
\item[\rm{(ii)}]$$\|b_{\lambda_j}-b_{\lambda_k}\|_{H^2}\leq 2|b_{\lambda_j}(\lambda_k)|^{1/2}$$
\item[\rm{(iii)}]$$\|\psi_j e_{\lambda_j}-\psi_k e_{\lambda_k}\|_{H^2}\leq 4|b_{\lambda_j}(\lambda_k)|^{1/2}$$
\item[\rm{(iv)}] $$\|\kappa_{\lambda_j}-\kappa_{\lambda_k}\|_{H^2}\leq |b_{\lambda_j}(\lambda_k)|^{1/2}\left(\frac{1}{1-|\lambda_j|^2}+\frac{1}{1-|\lambda_k|^2} \right)^{1/2}.$$
\end{enumerate}
\end{lemma}
\begin{proof}
We note first that 
$
b_{\lambda_j}(z)=(z-\lambda_j)\kappa_{\lambda_j}(z)
$
so that $b_{\lambda_j}=(S-\lambda_j)\kappa_{\lambda_j}$. Using the fact that $S^* \kappa_{\lambda}=\ol{\lambda}\kappa_{\lambda}$, we compute
\begin{align*}
\langle b_{\lambda_j},b_{\lambda_k}\rangle_{H^2} &=\langle (S-\lambda_j)\kappa_{\lambda_j},(S-\lambda_k)\kappa_{\lambda_k}\rangle_{H^2}\\
&=(1+\ol{\lambda_k}\lambda_j-|\lambda_j|^2-|\lambda_k|^2)\langle\kappa_{\lambda_j},\kappa_{\lambda_k}\rangle_{H^2} \\
&=\frac{1+\ol{\lambda_k}\lambda_j-|\lambda_j|^2-|\lambda_k|^2}{1-\ol{\lambda_j}\lambda_k}\\
&=1+\ol{\lambda_j}\frac{\lambda_k-\lambda_j}{1-\ol{\lambda_j}\lambda_k}+\ol{\lambda_k}\frac{\lambda_j-\lambda_k}{1-\ol{\lambda_j}\lambda_k}
\end{align*}
which gives (i). In particular, we find
$$
|1-\langle b_{\lambda_j},b_{\lambda_k}\rangle_{H^2}|\leq 2|b_{\lambda_j}(\lambda_k)|.
$$
But then
\begin{align*}
\|b_{\lambda_j}-b_{\lambda_k}\|_{H^2}^2&=\|b_{\lambda_j}\|_{H^2}^2+\|b_{\lambda_k}\|_{H^2}^2-\langle b_{\lambda_j},b_{\lambda_k}\rangle_{H^2} -\langle b_{\lambda_k},b_{\lambda_j}\rangle_{H^2} \\
&\leq 2|1-\langle b_{\lambda_j},b_{\lambda_k}\rangle_{H^2}|\\
&\leq 4|b_{\lambda_j}(\lambda_k)|
\end{align*}
which is (ii). For (iii), we note that
$$
e_{\lambda_j}=1+\ol{\lambda_j}b_{\lambda_j}
$$
and thus
$$
\psi_j e_{\lambda_j}-\psi_k e_{\lambda_k}=(\psi_j-\psi_k)+\ol{\lambda_j}\psi_j b_{\lambda_j}-\ol{\lambda_k}\psi_k b_{\lambda_k}=(\psi_j-\psi_k)+(\ol{\lambda_j}-\ol{\lambda_k})\theta.
$$
We find
\begin{align*}
\|\psi_j e_{\lambda_j}-\psi_k e_{\lambda_k}\|_{H^2}&\leq \|\psi_j-\psi_k\|_{H^2}+|\lambda_j-\lambda_k|\\
&=\|b_{\lambda_k}-b_{\lambda_j}\|_{H^2}+|\lambda_j-\lambda_k|.
\end{align*}
Note now that $|b_{\lambda_j}(\lambda_k)|\geq |\lambda_j-\lambda_k|/2$, so that by using (ii) we may write
$$
\|\psi_j e_{\lambda_j}-\psi_k e_{\lambda_k}\|_{H^2}\leq 2|b_{\lambda_j}(\lambda_k)|^{1/2}+2|b_{\lambda_j}(\lambda_k)|\leq 4|b_{\lambda_j}(\lambda_k)|^{1/2}
$$
and (iii) is established. Finally, we have
\begin{align*}
\|\kappa_{\lambda_j}-\kappa_{\lambda_k}\|_{H^2}^2&=\|\kappa_{\lambda_j}\|_{H^2}^2+\|\kappa_{\lambda_k}\|_{H^2}^2-\langle \kappa_{\lambda_j},\kappa_{\lambda_k}\rangle_{H^2} -\langle \kappa_{\lambda_k},\kappa_{\lambda_j}\rangle_{H^2}\\
&=\frac{1}{1-|\lambda_j|^2}+\frac{1}{1-|\lambda_k|^2}-\frac{1}{1-\ol{\lambda_j}\lambda_k}-\frac{1}{1-\ol{\lambda_k}\lambda_j}\\
&=\frac{\ol{\lambda_j}(\lambda_j-\lambda_k)}{(1-|\lambda_j|^2)(1-\ol{\lambda_j}\lambda_k)}+\frac{\ol{\lambda_k}(\lambda_k-\lambda_j)}{(1-|\lambda_k|^2)(1-\ol{\lambda_k}\lambda_j)}\\
&\leq |b_{\lambda_j}(\lambda_k)|\left(\frac{1}{1-|\lambda_j|^2}+\frac{1}{1-|\lambda_k|^2} \right)\\
\end{align*}
whence (iv) follows.
\end{proof}

Using these computations we can now establish an estimate that will be of use later.

\begin{lemma}\label{l-psi}
Let $\theta=b_{\lambda_1}\ldots b_{\lambda_N}$ and  set $\psi_j=\theta/b_{\lambda_j}$ for each $1\leq j \leq N$. Then, for every $1\leq j,k\leq N$ we have
$$
\|\psi_j-\psi_k \|_{H^\infty/\theta H^\infty}\leq \frac{5\sqrt{2}|b_{\lambda_j}(\lambda_k)|^{1/2}}{(1-\max \{|\lambda_j|,|\lambda_k| \}^2)^{1/2}}.
$$
\end{lemma}
\begin{proof}
If $\lambda_j=\lambda_k$, then the conclusion holds trivially, so we assume henceforth that $\lambda_j\neq\lambda_k$. We see that  
\begin{align*}
\|\psi_j-\psi_k \|_{H^\infty/\theta H^\infty}&= \inf\{\|\psi_j-\psi_k +\theta f \|_{H^\infty}: f\in H^\infty \}\\
&=\inf\left\{\left\|\frac{\theta}{b_{\lambda_j} b_{\lambda_k}}(b_{\lambda_k}-b_{\lambda_j} +b_{\lambda_j} b_{\lambda_k} f) \right\|_{H^\infty}: f\in H^\infty \right\}\\
&=\inf\{\|b_{\lambda_k}-b_{\lambda_j} +b_{\lambda_j} b_{\lambda_k} f \|_{H^\infty}: f\in H^\infty \}\\
&=\|b_{\lambda_k}-b_{\lambda_j}\|_{H^\infty /b_{\lambda_j} b_{\lambda_k}H^\infty }.
\end{align*}
Set $\phi=b_{\lambda_j} b_{\lambda_k}$.
By Lemma \ref{l-lcmspan}, we have that 
$$
H(\phi)=\ker b_{\lambda_j}(S(\phi))\vee\ker b_{\lambda_k}(S(\phi)) .
$$
The subspace $\ker b_{\lambda_i}(S(\phi))$ is one-dimensional for $i\in \{j,k\}$, and in fact it is spanned by
$
\phi e_{\lambda_i}/b_{\lambda_i} .
$
Therefore, any $h\in H(\phi)$ can be written as
$$
h=a_j \phi e_{\lambda_j}/b_{\lambda_j}+a_k \phi e_{\lambda_k}/b_{\lambda_k}
$$
for some $a_j,a_k \in \C$.
In particular, we see that
$$
h(\lambda_j)=a_j b_{\lambda_k}(\lambda_j)
$$
and
$$
h(\lambda_k)=a_k b_{\lambda_j}(\lambda_k).
$$
We get
\begin{align*}
\|(b_{\lambda_k}-b_{\lambda_j})(S(\phi))h \|&=\|b_{\lambda_k}(\lambda_j) a_j \phi e_{\lambda_j}/b_{\lambda_j}- b_{\lambda_j}(\lambda_k) a_k \phi e_{\lambda_k}/b_{\lambda_k}\|_{H^2}\\
&=\|h(\lambda_j)\phi e_{\lambda_j}/b_{\lambda_j}-h(\lambda_k)\phi e_{\lambda_k}/b_{\lambda_k}\|_{H^2}\\
&\leq |h(\lambda_j)| \|\phi e_{\lambda_j}/b_{\lambda_j}-\phi e_{\lambda_k}/b_{\lambda_k}\|_{H^2}+|h(\lambda_j)-h(\lambda_k)| \|\phi e_{\lambda_k}/b_{\lambda_k}\|_{H^2}\\
&\leq \|h\|_{H^2} (1-|\lambda_j|^2)^{-1/2} \|\phi e_{\lambda_j}/b_{\lambda_j}-\phi e_{\lambda_k}/b_{\lambda_k}\|_{H^2}\\
&+ (1-|\lambda_k|^2)^{1/2}|h(\lambda_j)-h(\lambda_k)|\\
&\leq \|h\|_{H^2} (1-|\lambda_j|^2)^{-1/2} \|\phi e_{\lambda_j}/b_{\lambda_j}-\phi e_{\lambda_k}/b_{\lambda_k}\|_{H^2}\\
&+ (1-|\lambda_k|^2)^{1/2}\|h\|_{H^2} \|\kappa_{\lambda_j}-\kappa_{\lambda_k}\|_{H^2}.
\end{align*}
By virtue of Lemma \ref{lemmacompkernel}, we find
\begin{align*}
\|(b_{\lambda_k}-b_{\lambda_j})(S(\phi))h \|& \leq 5 |b_{\lambda_j}(\lambda_k)|^{1/2}\|h\| \left( \frac{1}{1-|\lambda_j|^2}+\frac{1}{1-|\lambda_k|^2}\right)^{1/2}\\
&\leq  \frac{5\sqrt{2}|b_{\lambda_j}(\lambda_k)|^{1/2}}{(1-\max \{|\lambda_j|,|\lambda_k| \}^2)^{1/2}}\|h\|.
\end{align*}
Since $h\in H(\theta)$ was arbitrary, we find
$$
\|b_{\lambda_k}-b_{\lambda_j}\|_{H^\infty/\phi H^\infty}=\|(b_{\lambda_k}-b_{\lambda_j})(S(\phi))\|\leq  \frac{5\sqrt{2}|b_{\lambda_j}(\lambda_k)|^{1/2}}{(1-\max \{|\lambda_j|,|\lambda_k| \}^2)^{1/2}}
$$
and in view of the equality 
$$
\|\psi_j-\psi_k \|_{H^\infty/\theta H^\infty}=\|b_{\lambda_k}-b_{\lambda_j}\|_{H^\infty /\phi H^\infty }
$$
the proof is complete.
\end{proof}

For the next step, we introduce some notation. Given $\theta=b_{\lambda_1}\cdots b_{\lambda_N}$, we set $\alpha_0=1$ and
$$
\alpha_k=b_{\lambda_1}\ldots b_{\lambda_k}
$$
for $1\leq k \leq N$. These functions allow us to pick a basis for the underlying Hilbert space that is well-adapted to our purpose, as is made clear in the following proposition. First we need an easy lemma.

\begin{lemma}\label{l-cyclic}
Let $T\in B(\hil)$ be a multiplicity-free operator of class $C_0$ with minimal function $\theta=b_{\lambda_1}\ldots b_{\lambda_N}$. Let $\xi\in \hil$ be a unit vector which is also cyclic for $T$. Then, the set $\{\alpha_k(T)\xi:0\leq k \leq N-1\}$ is a basis for $\hil$.
\end{lemma}
\begin{proof}
It follows from Theorem \ref{classifC0} that $\hil$ has dimension $N$, and thus it suffices to show that the set $\{\alpha_k(T)\xi:0\leq k \leq N-1\}$ is linearly independent. Assume therefore that there are some $c_0,\ldots,c_{N-1}\in \C$ such that
$$
\sum_{k=0}^{N-1}c_k \alpha_k(T)\xi=0.
$$
Because $\xi$ is assumed to be cyclic for $T$, any vector $h\in \hil$ can be written as $h=p(T)\xi$ for some polynomial $p$. Therefore, the relation above implies that 
$$
\sum_{k=0}^{N-1}c_k \alpha_k(T)=0
$$
and thus $\theta$ divides $\sum_{k=0}^{N-1}c_k \alpha_k$. In particular, $\sum_{k=0}^{N-1}c_k \alpha_k$ must vanish at $\lambda_1$ and we find
$$
c_0=\sum_{k=0}^{N-1}c_k \alpha_k(\lambda_1)=0
$$
whence $\theta$ divides $c_1 \alpha_1+\ldots+c_{N-1}\alpha_{N-1}$. This implies that $\theta/b_{\lambda_1}$ divides 
$$
c_1+\sum_{k=2}^{N-1}c_k \alpha_k/b_{\lambda_1}
$$
so this last function vanishes at $\lambda_2$, which yields $c_1=0$. Proceeding inductively, we find $c_0=c_1=\ldots=c_{N-1}=0$.
\end{proof}

We now establish a crucial estimate on the angle between the different elements of the basis appearing in the above lemma.

\begin{proposition}\label{p-dilationangle}
Let $T\in B(\hil)$ be a multiplicity-free operator of class $C_0$ with minimal function $\theta=b_{\lambda_1}\ldots b_{\lambda_N}$. Define
$$
\eta=\sup_{1\leq j,k\leq N}\frac{|b_{\lambda_j}(\lambda_k)|^{1/2}}{(1-\max \{|\lambda_j|,|\lambda_k| \}^2)^{1/2}}.
$$
Assume that $\|\psi_N(T)\|> \beta+5\sqrt{2}\eta$ for some $0<\beta< 1$. Then, there exists a unit vector $\xi\in \hil$ which is cyclic for $T$ with the property that
$$
\|\phi(T)\xi\|\geq \beta
$$
for every inner divisor $\phi$ of $\theta$ and
$$
|\langle \alpha_j(T)\xi,\alpha_k(T)\xi\rangle|\leq \sqrt{1-\beta^2}\|\alpha_j(T) \xi\| \|\alpha_k(T)\xi\|
$$
for $0\leq j<k\leq N-1$.
\end{proposition}
\begin{proof}
Since $\|\psi_N(T)\|>\beta+5\sqrt{2} \eta$, we may invoke Theorem \ref{maxvector} to find a unit vector $\xi \in \hil$ which is cyclic for $T$ and such that $\|\psi_N(T)\xi\|>\beta+5\sqrt{2} \eta$. Using Lemma \ref{l-psi}, for every $1\leq j \leq N-1$ we find that 
\begin{align*}
\|\psi_j(T)\xi\|&\geq \|\psi_N(T)\xi\|-\|(\psi_j-\psi_N)(T)\xi\|\\
&\geq \beta+5\sqrt{2}\eta-\|\psi_j-\psi_N\|_{H^\infty/\theta H^\infty}\\
&\geq \beta.
\end{align*}
Given an inner divisor $\phi$ of $\theta$, there always exists some index $1\leq j \leq N$ for which $\psi_j= \phi' \phi $ for some inner function $\phi'\in H^\infty$. Thus,
$$
\|\phi(T)\xi\| \geq \|\phi'(T)\phi(T)\xi\|= \|\psi_j(T)\xi\|\geq \beta,
$$
and the first statement is established.
Let us now consider $U:\kil\to \kil$ the minimal unitary dilation of $T:\hil\to \hil$ (see \cite{NF2} for details).  For every function $f\in H^\infty$ and every vector $h\in \hil$, we have that 
$$
(f(U)-f(T))h=(f(U)-P_{\hil}f(U))h=P_{\kil\ominus \hil}f(U)h
$$
so that
$$
\| (f(U)-f(T))h\|^2 =\|f(U)h\|^2-\|P_{\hil}f(U)h\|^2=\|f(U)h\|^2-\|f(T)h\|^2
$$
and thus
\begin{align*}
\left\| \left(\frac{\theta}{\alpha_k}(U)-\frac{\theta}{\alpha_k}(T)\right)\alpha_j(T)\xi\right\|^2 &= \left\|\frac{\theta}{\alpha_k}(U)\alpha_j(T)\xi\right\|^2-\left\|\frac{\theta \alpha_j }{\alpha_k}(T)\xi\right\|^2\\
&\leq \|\alpha_j(T)\xi\|^2-\beta^2\\
&\leq (1-\beta^2) \|\alpha_j(T)\xi\|^2
\end{align*}
whenever $0\leq j<k\leq N-1$. 
In addition, note that
\begin{align*}
\left\langle \frac{\theta}{\alpha_k}(T)\alpha_j(T)\xi,\frac{\theta}{\alpha_k}(U)\alpha_k(T)\xi \right\rangle&=\left\langle \frac{\theta}{\alpha_k}(T)\alpha_j(T)\xi,P_{\hil}\frac{\theta}{\alpha_k}(U)\alpha_k(T)\xi \right\rangle\\
&=\left\langle \frac{\theta}{\alpha_k}(T)\alpha_j(T)\xi,\frac{\theta}{\alpha_k}(T)\alpha_k(T)\xi \right\rangle\\
&=0
\end{align*}
since $\theta(T)=0$.

Using with these facts, we find 
\begin{align*}
\left|\langle \alpha_j(T)\xi,\alpha_k(T)\xi \rangle\right| &=\left|\left\langle \frac{\theta}{\alpha_k}(U) \alpha_j(T)\xi,\frac{\theta}{\alpha_k}(U)\alpha_k(T)\xi \right\rangle\right|\\
&=\left|\left\langle \left(\frac{\theta}{\alpha_k}(U) -\frac{\theta}{\alpha_k}(T)\right)\alpha_j(T)\xi,\frac{\theta}{\alpha_k}(U)\alpha_k(T)\xi \right\rangle\right|\\
&\leq \left\| \left(\frac{\theta}{\alpha_k}(U) -\frac{\theta}{\alpha_k}(T)\right)\alpha_j(T)\xi\right\| \left\| \alpha_k(T) \xi\right\|\\
&\leq \sqrt{1-\beta^2}\|\alpha_j(T)\xi\|\|\alpha_k(T)\xi\|
\end{align*}
where we used the fact that $(\theta/\alpha_k)(U)$ is unitary. The proof is complete.
\end{proof}
This proposition was the last ingredient needed to prove the principal technical tool of this section.

\begin{theorem}\label{t-simalg}
Let $T_1\in B(\hil_1), T_2\in B(\hil_2)$ be multiplicity-free operators of class $C_0$ with minimal function $\theta=b_{\lambda_1}\ldots b_{\lambda_N}$. Define
$$
\eta=\sup_{1\leq j,k\leq N}\frac{|b_{\lambda_j}(\lambda_k)|^{1/2}}{(1-\max \{|\lambda_j|,|\lambda_k| \}^2)^{1/2}}.
$$
Assume that 
$$
\|\psi_N(T_1)\|>  \beta_1+5\sqrt{2}\eta.
$$
and
$$
\|\psi_N(T_2)\|> \beta_2+5\sqrt{2}\eta
$$ 
for some constants $\beta_1,\beta_2$ satisfying
$$
\sqrt{1-\frac{1}{(N-1)^2}}<\beta_1,\beta_2<1.
$$
Then, there exists an invertible operator $X:\hil_1\to \hil_2$ such that $XT_1=T_2 X$ and 
$$
\max\{\|X\|,\|X^{-1}\| \}\leq C(\beta_1,\beta_2,N)
$$
where $ C(\beta_1,\beta_2,N)>0$ is a constant depending only on $\beta_1,\beta_2$ and $N$.
\end{theorem}
\begin{proof}
By Proposition \ref{p-dilationangle}, for every $i=1,2$ we can find a unit cyclic vector $\xi_i\in \hil_i$ with the property that
$$
|\langle \alpha_j(T_i)\xi_i,\alpha_k(T_i)\xi_i\rangle|\leq\sqrt{1-\beta_i^2}\|\alpha_j(T_i) \xi_i\| \|\alpha_k(T_i)\xi_i\|
$$
and
$$
\|\alpha_j(T_i)\xi_i\|\geq \beta_i
$$
whenever $0\leq j<k\leq N-1$. 
Given $c_0,\ldots, c_{N-1}\in \C$, let us set
$$
X\left(\sum_{k=0}^{N-1}c_k \alpha_k(T_1)\xi_1\right)=\sum_{k=0}^{N-1}c_k \alpha_k(T_2) \xi_2.
$$
By Lemma \ref{l-cyclic}, this defines an invertible operator $X:\hil_1\to \hil_2$ with $X\xi_1=\xi_2$. Arguing as in the proof of Lemma \ref{l-cyclic}, we see that
$$
u(T_1)\xi_1=\sum_{k=0}^{N-1}c_k \alpha_k(T_1)\xi_1
$$
holds if and only if $\theta$ divides $\sum_{k=0}^{N-1}c_k \alpha_k-u$, which in turn holds if and only if
$$
u(T_2)\xi_2=\sum_{k=0}^{N-1}c_k \alpha_k(T_2)\xi_2.
$$
Therefore, we find
$$
Xu(T_1)\xi_1=u(T_2)\xi_2=u(T_2)X\xi_1
$$
for any function $u\in H^\infty$. Now, every $h\in \hil_1$ can be written as $h=p(T_1)\xi_1$ for some polynomial $p$, so that
$$
XT_1h=XT_1 p(T_1)\xi_1=T_2p(T_2)X\xi_1=T_2Xp(T_1)\xi_1=T_2 X h
$$
for every $h\in \hil_1$, whence $XT_1=T_2X$. It only remains to estimate the norm of $X$ and $X^{-1}$. 
Given $c_0,\ldots,c_{N-1}\in \C$, we see that
\begin{align*}
&\sum_{k=0}^{N-1}|c_k|^2 \|\alpha_k(T_1)\xi_1\|^2+2\Re\left(\sum_{0\leq j<k\leq N-1} c_j\ol{c_k}\langle \alpha_j(T_1)\xi_1,\alpha_k(T_1)\xi_1\rangle\right)\\
&\geq \sum_{k=0}^{N-1}|c_k|^2 \|\alpha_k(T_1)\xi_1\|^2-2\sum_{0\leq j<k\leq N-1} |c_j| |c_k| |\langle \alpha_j(T_1)\xi_1,\alpha_k(T_1)\xi_1\rangle|\\
&\geq \sum_{k=0}^{N-1}|c_k|^2 \|\alpha_k(T_1)\xi_1\|^2
-2\sum_{0\leq j<k\leq N-1} |c_j| |c_k| \sqrt{1-\beta_1^2}\|\alpha_j(T_1)\xi_1\| \|\alpha_k(T_1)\xi_1\|\\
&=\left(1-(N-1)\sqrt{1-\beta_1^2} \right)\sum_{k=0}^{N-1}|c_k|^2 \|\alpha_k(T_1)\xi_1\|^2\\
&+\sqrt{1-\beta_1^2}\left((N-1)\sum_{k=0}^{N-1}|c_k|^2 \|\alpha_k(T_1)\xi_1\|^2-2\sum_{0\leq j<k\leq N-1} |c_j| |c_k| \|\alpha_j(T_1)\xi_1\| \|\alpha_k(T_1)\xi_1\|\right)\\
\end{align*}
The last bracketed term being positive, the calculation above gives
$$
\left\|\sum_{k=0}^{N-1}c_k \alpha_k(T_1)\xi_1\right\|^2
\geq \left(1-(N-1)\sqrt{1-\beta_1^2} \right)\sum_{k=0}^{N-1}|c_k|^2 \|\alpha_k(T_1)\xi_1\|^2.
$$
Thus,
\begin{align*}
\left\|\sum_{k=0}^{N-1}c_k \alpha_k(T_1)\xi_1\right\|^2
&\geq \left(1-(N-1)\sqrt{1-\beta_1^2} \right)\sum_{k=0}^{N-1}|c_k|^2 \|\alpha_k(T_1)\xi_1\|^2\\
&\geq \left(1-(N-1)\sqrt{1-\beta_1^2} \right)\sum_{k=0}^{N-1}|c_k|^2 \beta_1^2\\
&\geq \frac{\beta_1^2}{N}\left(1-(N-1)\sqrt{1-\beta_1^2} \right)\left(\sum_{k=0}^{N-1}|c_k|\right)^2 \\
&\geq \frac{\beta_1^2}{N}\left(1-(N-1)\sqrt{1-\beta_1^2} \right)\left\|\sum_{k=0}^{N-1}c_k\alpha_k\right\|_{H^\infty}^2\\
&\geq \frac{\beta_1^2}{N}\left(1-(N-1)\sqrt{1-\beta_1^2} \right)\left\|\sum_{k=0}^{N-1}c_k\alpha_k(T_2)\xi_2\right\|^2
\end{align*}
where we used $\|\alpha_k\|_{H^\infty}=1$.
By symmetry, we also have
$$
\left\|\sum_{k=0}^{N-1}c_k \alpha_k(T_2)\xi_2\right\|^2\geq \frac{\beta_2^2}{N}\left(1-(N-1)\sqrt{1-\beta_2^2} \right)\left\|\sum_{k=0}^{N-1}c_k\alpha_k(T_1)\xi_1\right\|^2.
$$
This shows that 
$$
\|X\|^2\leq \left(\frac{\beta_1^2}{N}\left(1-(N-1)\sqrt{1-\beta_1^2} \right)\right)^{-1}
$$
$$
\|X^{-1}\|^2\leq \left(\frac{\beta_2^2}{N}\left(1-(N-1)\sqrt{1-\beta_2^2} \right)\right)^{-1}
$$
and the proof is complete.
\end{proof}

The next corollary is our main result. 

\begin{corollary}\label{c-simalg}
Let $T\in B(\hil)$ be a multiplicity-free operator of class $C_0$ with minimal function $\theta=b_{\lambda_1}\ldots b_{\lambda_N}$. Define
$$
\eta=\sup_{1\leq j,k\leq N}\frac{|b_{\lambda_j}(\lambda_k)|^{1/2}}{(1-\max \{|\lambda_j|,|\lambda_k| \}^2)^{1/2}}.
$$
Consider $$\Psi:H^\infty(T)\to H^\infty (S(\theta))$$ such that $\Psi(u(T))=u(S(\theta))$ for every $u\in H^\infty$. Assume that $\Psi$ is bounded and that
$$
\sqrt{1-\frac{1}{(N-1)^2}}<\beta<1
$$
where $\beta=1/\|\Psi\|-5\sqrt{2}\eta$.
Then, there exists an invertible operator $X:\hil\to H(\theta)$ such that $XT=S(\theta) X$ and 
$$
\max\{\|X\|,\|X^{-1}\| \}\leq C(\Psi,N)
$$
where $C(\Psi,N)>0$ is a constant depending only on $\Psi$ and $N$.
\end{corollary}
\begin{proof}
Notice first that
\begin{align*}
\|\psi_N\|_{H^\infty/\theta H^\infty}&=\inf\{\|\psi_N+\theta f\|_{H^\infty} :f\in H^\infty\}\\
&=\inf \{ \|1+(\theta/\psi_N)f\|_{H^\infty}:f\in H^\infty\}\\
&=\inf\{ \|1+b_{\lambda_N}f\|_{H^\infty}:f\in H^\infty\}\\
&\geq \inf\{ \|(1+b_{\lambda_N}f)(\lambda_N)\|_{H^\infty}:f\in H^\infty\}\\
&=1
\end{align*}
and thus 
$$
\|\psi_N(S(\theta))\|=\|\psi_N\|_{H^\infty/\theta H^\infty}=1,
$$
By the contractive property of the functional calculus, we have
$$
\|\psi_N(S(\theta))\|=\|\psi_N\|_{H^\infty/\theta H^\infty}\geq \|\psi_N(T)\|\geq \|\Psi\|^{-1}\|\psi_N(S(\theta))\|= \|\Psi\|^{-1}=\beta+5\sqrt{2}\eta.
$$
The result now follows directly from Theorem \ref{t-simalg}.
\end{proof}

Notice that the inequality
$$
\sqrt{1-\frac{1}{(N-1)^2}}<\beta=\frac{1}{\|\Psi\|}-5\sqrt{2}\eta 
$$
implies that $\eta$ cannot be too large. This restricts the positions of the roots $\lambda_1,\ldots,\lambda_N$ relative to each other and to the boundary of the disc. On the other hand, without these inequalities the existence of the isomorphism $\Psi$ is clearly a necessary condition for similarity between $T$ and $S(\theta)$.
%

The upcoming corollary deals with the simpler two-dimensional case where the statements are neater and fewer conditions are needed. In fact, a careful look at the proofs of this section shows that the assumptions involving $\eta$ in the previous results arise solely because of the need to obtain lower bounds on $\|\phi(T)\xi\|$ for every inner divisor $\phi$ of  $\theta$ from a given lower bound on $\|\psi_N(T)\xi\|$. In the two dimensional case, there is a trick to obtaining these lower bounds without imposing any condition on $\eta$; it is the essence of the next lemma.

\begin{lemma}\label{l-blambda}
Let $T\in B(\hil)$ be a contraction and let $\lambda_1,\lambda_2\in \D$. If $\xi\in \hil$ is a unit vector such that $\|b_{\lambda_1}(T)\xi\|\geq \beta$, then
$$
\|b_{\lambda_2}(T)\xi\|\geq \sqrt{(1-|\mu|)^2-(1-|\mu|^2)(1-\beta^2)}-|\mu|,
$$
where $\mu=b_{\lambda_2}(\lambda_1)$.
\end{lemma}
\begin{proof}

A direct calculation shows that for any contraction $R\in B(\hil)$ and any $\mu\in \D$, we have
$$
I-b_{\mu}(R)^*b_{\mu}(R)=(1-|\mu|^2)(1-\mu R^*)^{-1}(I-R^*R)(1-\ol{\mu}R)^{-1}
$$
and thus
$$
\langle (I-b_{\mu}(R)^*b_{\mu}(R))(1-\ol{\mu}R)h,(1-\ol{\mu}R)h\rangle=(1-|\mu|^2)\langle (I-R^*R)h,h\rangle
$$
for every $h\in \hil$. This implies that
\begin{align*}
\|b_{\mu}(R)(I-\ol{\mu}R)h\|^2&=\|(I-\ol{\mu}R)h\|^2-\langle (I-b_{\mu}(R)^*b_{\mu}(R))(1-\ol{\mu}R)h,(1-\ol{\mu}R)h\rangle\\
&=\|(I-\ol{\mu}R)h\|^2-(1-|\mu|^2)\langle (I-R^*R)h,h\rangle\\
&=\|(I-\ol{\mu}R)h\|^2-(1-|\mu|^2)(\|h\|^2-\|Rh\|^2)\\
& \geq (\|h\|-|\mu| \|Rh\|)^2-(1-|\mu|^2)(\|h\|^2-\|Rh\|^2)\\
&\geq (1-|\mu|)^2\|h\|^2-(1-|\mu|^2)(\|h\|^2-\|Rh\|^2).
\end{align*}
Note now that if $\mu=b_{\lambda_1}(\lambda_2)$, then $b_{\mu}=b_{\lambda_2}\circ b^{-1}_{\lambda_1}$. Applying the previous inequality with $R=b_{\lambda_1}(T)$, $h=\xi$ and $\mu=b_{\lambda_1}(\lambda_2)$, we find
$$
\|b_{\lambda_2}(T)(I-\ol{\mu}b_{\lambda_1}(T))\xi\|^2\geq (1-|\mu|)^2\|\xi\|^2-(1-|\mu|^2)(\|\xi\|^2-\|b_{\lambda_1}(T)\xi\|^2).
$$
By assumption, this becomes
$$
\|b_{\lambda_2}(T)(I-\ol{\mu}b_{\lambda_1}(T))\xi\|^2\geq (1-|\mu|)^2-(1-|\mu|^2)(1-\beta^2).
$$
Finally, we have
\begin{align*}
\|b_{\lambda_2}(T)\xi\|&\geq \|b_{\lambda_2}(T)(I-\ol{\mu}b_{\lambda_1}(T))\xi\|-|\mu|\|b_{\lambda_2}(T)b_{\lambda_1}(T)\xi\|\\
&\geq \sqrt{(1-|\mu|)^2-(1-|\mu|^2)(1-\beta^2)}-|\mu|
\end{align*}
and this finishes the proof.
\end{proof}

\begin{corollary}\label{c-dim2}
Let $T\in B(\hil)$ be a multiplicity-free operator of class $C_0$ with minimal function $\theta=b_{\lambda_1} b_{\lambda_2}$. 
Consider $$\Psi:H^\infty(T)\to H^\infty (S(\theta))$$ such that $\Psi(u(T))=u(S(\theta))$ for every $u\in H^\infty$. Assume that $\Psi$ is bounded.
Then, there exists an invertible operator $X:\hil\to H(\theta)$ such that $XT=S(\theta) X$ and 
$$
\max\{\|X\|,\|X^{-1}\| \}\leq C(\Psi)
$$
where $C(\Psi)>0$ is a constant depending only on $\Psi$.
\end{corollary}
\begin{proof}
Set $\mu=b_{\lambda_1}(\lambda_2)$. Fix some number $0<\e< \|\Psi\|^{-1} $ and set $\beta=\|\Psi\|^{-1}-\e.$
Consider
$$
\beta'=\sqrt{(1-|\mu|)^2-(1-|\mu|^2)(1-\beta^2)}- |\mu|
$$
and
$$
\gamma=1-2 |\mu|.
$$
There exits a positive number $r>0$ depending only on $\Psi$ such that
$\beta'> \beta/2$ and $\gamma> 1/2$ if $|\mu|<r$.
We distinguish two cases. First, we assume that $|\mu|\geq r$. This corresponds to the case of ``uniformly separated roots", for which the existence of an invertible operator $X$ with the required properties is well-known (see for instance Proposition 3.2 of \cite{clouatre}). We turn now to the case where
$|\mu|< r$. The unit vector $\zeta=\kappa_{\lambda_1}/\| \kappa_{\lambda_1}\|_{H^2}\in H(\theta)$ is cyclic for $S(\theta)$ and
$$
\|b_{\lambda_2}(S(\theta))\zeta\|=1
$$
(these facts are easily verified, and they can be found in Lemma 3.2.1 and Corollary 3.2.4 of \cite{arveson}.)
In particular, we have
$$
\|b_{\lambda_2}(S(\theta))\|=1
$$
and thus
$$
\|b_{\lambda_2}(T)\|\geq \|\Psi\|^{-1}\|b_{\lambda_2}(S(\theta))\|=\|\Psi\|^{-1}.
$$
By Theorem \ref{maxvector}, we can find a unit vector $\xi\in \hil$ which is cyclic for $T$ with the property that
$$
\|b_{\lambda_2}(T)\xi\|\geq \beta.
$$
Note also that Lemma \ref{l-blambda} implies that
$$
\|b_{\lambda_1}(T)\xi\|\geq \beta'
$$
and
$$
\|b_{\lambda_1}(S(\theta))\zeta\|\geq \gamma.
$$
Repeating the argument done in the proof Proposition \ref{p-dilationangle} we find
$$
|\langle \xi,b_{\lambda_1}(T)\xi\rangle|\leq \sqrt{1-\beta^2}\|\xi \|\|b_{\lambda_1}(T)\xi\|
$$
along with
$$
\langle \zeta,b_{\lambda_1}(S(\theta))\zeta\rangle=0.
$$
By Lemma \ref{l-cyclic}, the set $\{\xi,b_{\lambda_1}(T)\xi\}$ forms a basis for $\hil$ while $\{\zeta,b_{\lambda_1}(S(\theta))\zeta\}$ forms a basis for $H(\theta)$. As in the proof of Theorem \ref{t-simalg}, we define $X:\hil\to H(\theta)$ as
$$
X(c_0 \xi+c_1 b_{\lambda_1}(T)\xi)=c_0 \zeta+ c_1 b_{\lambda_1}(S(\theta))\zeta
$$
where $c_0,c_1\in \C$. This operator satisfies $XT=S(\theta)X$. 
We have
\begin{align*}
&\left\|c_0 \xi+c_1 b_{\lambda_1}(T)\xi \right\|^2\\
&\geq |c_0|^2+|c_1|^2\|b_{\lambda_1}(T)\xi \|^2-2|c_0| |c_1|\sqrt{1-\beta^2}\|\xi \|\|b_{\lambda_1}(T)\xi\|\\
&\geq (1-\sqrt{1-\beta^2})( |c_0|^2+|c_1|^2\|b_{\lambda_1}(T)\xi \|^2)\\
&\geq \beta'^{2}(1-\sqrt{1-\beta^2})(|c_0|^2+|c_1|^2)\\
&\geq \frac{1}{2}\beta'^{2}(1-\sqrt{1-\beta^2})\|c_0 \zeta+c_1 b_{\lambda_1}(S(\theta))\zeta \|^2
\end{align*}
and 
\begin{align*}
\left\|c_0 \xi+c_1 b_{\lambda_1}(T)\xi \right\|^2&\leq 2(\|c_0 \xi\|^2+\|c_1  b_{\lambda_1}(T)\xi\|^2)\\
&\leq2 (|c_0|^2+|c_1|^2)\\
&\leq2 \gamma^{-2}\|c_0 \zeta+c_1 b_{\lambda_1}(S(\theta))\zeta \|^2.
\end{align*}
Consequently, we find
$$
\|X\|^2\leq  \left(\frac{1}{2}\beta'^{2}(1-\sqrt{1-\beta^2})\right)^{-1}
$$
$$
\|X^{-1}\|^2\leq 2\gamma^{-2}
$$
and using the fact that $\beta'> \beta/2$ and $\gamma>1/2$, we get
$$
\|X\|^2\leq  \left(\frac{1}{8}\beta^{2}(1-\sqrt{1-\beta^2})\right)^{-1}
$$
$$
\|X^{-1}\|^2\leq 8
$$
which completes the proof.
\end{proof}

\section{The generalized Carleson condition}

Let us first recall a definition. A sequence of distinct points $\{\lambda_j\}_j\subset \D$ is said to
satisfy the Carleson condition  if
\begin{equation}\label{unifsep}
\inf_{k} \prod_{j\neq k} \left|\frac{\lambda_{j}-\lambda_{k}}{1-\ol{\lambda_j}\lambda_k}\right|>0.
\end{equation}
The main focus of the work done in \cite{clouatre} was the case where the minimal function of a multiplicity-free operator of class $C_0$ is a Blaschke product with zeros forming such a sequence.
The purpose of this section is to indicate that the following result (see \cite{clouatre}) carries over to a more general setting. Although the result itself is known, we hope our approach (more precisely Lemma \ref{group}) might hold some independent interest.

\begin{theorem}\label{equivalence2}
Let $\{\lambda_j\}_j\subset \D$ be a sequence of distinct points and let $\theta \in H^\infty$ be the corresponding Blaschke product. The following statements are
equivalent:
\begin{enumerate}
\item[\rm{(i)}] $\{\lambda_j\}_j\subset \D$ satisfies the Carleson condition
\item[\rm{(ii)}] every multiplicity-free operator $T$ of class $C_0$ with minimal function $\theta$ is similar to $S(\theta)$.
\end{enumerate}
\end{theorem}
The relevant concept for us will be the following (see \cite{nikOP} for more details). Let $\{\theta_n\}_n\subset H^\infty$ be a sequence of inner functions. It will be convenient throughout to use the following notation: given a subset $A \subset \N$, we write
$$
\theta_{A}=\prod_{n\in A}\theta_n.
$$
We say that the sequence $\{\theta_n\}_n$ satisfies the generalized Carleson condition (with constant $C>0$) if for any
finite set $A \subset \N$ there are functions $f_{A},
g_{A}\in H^\infty$ satisfying
$$
f_{A}\theta_{A}+g_{A}\frac{\theta}{\theta_{A}}=1
$$
with $\|f_{A}\|_{H^\infty}\leq C$ and $\|g_{A}\|_{H^\infty}\leq C$.
It is a straightforward consequence of the definition that the functions $\theta_A$ and $\theta_B$ have no common inner divisor if $A$ and $B$ are disjoint and at least one of them is finite. Moreover, it is well-known that in the case where $\theta_n$ is simply a Blaschke factor, the generalized Carleson condition is equivalent to the classical Carleson condition (\ref{unifsep}) (see Lemma 3.2.18 in \cite{nikOP}).


Let $\{\theta_n\}_n \subset H^\infty$ be a sequence of inner functions satisfying the generalized Carleson condition with constant $C>0$, and set $\theta=\prod_n \theta_n\in H^\infty$. We have for every finite or cofinite subset $A\subset \N$ a pair of functions $f_A,g_A\in H^\infty$ satisfying
\begin{equation}\label{coronaAB}
f_{A}\theta_A+g_{A}\frac{\theta}{\theta_{A}}=1
\end{equation}
and
$\|f_A\|_{H^\infty}\leq C, \|g_A\|_{H^\infty}\leq C$. Set $\phi_A=g_{A}\theta/\theta_A$ and notice that
\begin{equation}\label{normphi}
\|\phi_A\|_{H^\infty}\leq C.
\end{equation}
We use the usual notation for the symmetric difference of two sets:
$$
A\triangle B=(A\setminus B) \cup (B\setminus A).
$$

\begin{lemma}\label{group}
Let $T\in B(\hil)$ be an operator of class $C_0$ with minimal function $\theta=\prod_n \theta_n$ where $\{\theta_n\}_n$ satisfies the generalized Carleson condition. For
every finite or cofinite subset $A\subset \N$, define
$$
k_A=2\phi_A(T)-I \in B(\hil).
$$
Then $\phi_A (T)\phi_B(T)=\phi_{A\cap B}(T)$, $k_A=k_A^{-1}$ and $k_A k_B=k_{\N\setminus (A\triangle B)}$. Moreover,
$$G=\{k_A:A\subset \N \text{ finite or cofinite }\}$$ is an abelian multiplicative subgroup of $B(\hil)$.
\end{lemma}
\begin{proof}
We first need to check that $k_A$ is well-defined since the function $g_A$ is not uniquely determined. Assume that $g_1$ and $g_2$ are two functions in $H^\infty$ which are candidates for $g_A$, meaning that they satisfy
$$
f_{1}\theta_A+g_{1}\frac{\theta}{\theta_{A}}=1
$$
and
$$
f_{2}\theta_A+g_{2}\frac{\theta}{\theta_{A}}=1
$$
for some functions $f_1,f_2\in H^\infty$.
Then, we find
$$
(g_1-g_2)\frac{\theta}{\theta_{A}}=(f_2-f_1)\theta_A.
$$
We see that $\theta_A$ must divide the left-hand side, and thus it must divide $g_1-g_2$ seeing as it has no common inner factor with $\theta/\theta_{A}$. Therefore, $\theta$ divides $(g_1-g_2)\theta/\theta_{A}$ and we have $(g_1-g_2)(\theta/\theta_A)(T)=0$. In other words, $\phi_A(T)$ and $k_A$ are well-defined. Notice now that
$$
\phi_A^2=\left(g_A\frac{\theta}{\theta_{A}}\right)^2=g_A\frac{\theta}{\theta_{A}} \left(1-f_A\theta_A\right)=\phi_A-f_A g_A \theta
$$
which implies that $\phi_A(T)^2=\phi_A(T)$.
A straightforward calculation now yields that $k_A^{-1}=k_A$.
Using (\ref{coronaAB}), we find that
\begin{equation}\label{eq1phiAB}
1=f_A f_B \theta_A \theta_B+f_A g_B \theta_A
\frac{\theta}{\theta_{B}}+f_B g_A \theta_B
\frac{\theta}{\theta_{A}}+\phi_A \phi_B\\
=h \theta_{A\cap B}+\phi_A \phi_B
\end{equation}
for some $h\in H^\infty$. On the other hand, we have
\begin{equation}\label{eq2phiAB}
f_{A\cap B}\theta_{A\cap B}+\phi_{A\cap B}=1.
\end{equation}
Notice now that
\begin{align*}
\theta_A \theta_B&=\prod_{n\in A}\theta_n \prod_{n\in  B}\theta_n\\
&=\prod_{n\in B\setminus A}\theta_n\prod_{n\in A\setminus B}\theta_n\left(\prod_{n\in A\cap B}\theta_n\right)^2\\
&=\prod_{n\in A\cup B}\theta_n\prod_{n\in A\cap B}\theta_n\\
&=\theta_{A \cap B}\theta_{A \cup B}.
\end{align*}
Consequently, using (\ref{eq1phiAB}) and (\ref{eq2phiAB}) we find
\begin{align*}
\frac{\theta}{\theta_{A\cap B}}\left( g_A g_B \frac{\theta}{\theta_{A \cup B}}-g_{A\cap B}\right)&=g_A g_B \frac{\theta}{\theta_{A}}\frac{\theta}{\theta_{B}} -g_{A\cap B}\frac{\theta}{\theta_{A\cap B}}\\
&=\phi_A \phi_B-\phi_{A\cap B}\\
&=(f_{A\cap B}-h)\theta_{A\cap B}.
\end{align*}
Therefore, $\theta_{A\cap B}$ divides the left-hand side, and as before since $\theta_{A\cap B}$ and $\theta/\theta_{A\cap B}$ have no common inner factor  we conclude that $\theta$ divides $\phi_A \phi_B-\phi_{A\cap B}$, which in turn implies
$$
\phi_A (T)\phi_B(T)=\phi_{A\cap B}(T).
$$
We now proceed to show the identity $k_A k_B=k_{\N\setminus (A\triangle B)}$ in a similar fashion.
First we note that
$$
k_A k_B=2(2\phi_A \phi_B-\phi_A-\phi_B+1)(T)-I
$$
so we need to establish
$$
(2\phi_A \phi_B-\phi_A-\phi_B+1)(T)=\phi_{\N\setminus (A\triangle B)}(T)
$$
which is equivalent to showing that the function $\theta$ divides the function
$$
2\phi_A \phi_B-\phi_A-\phi_B+1-\phi_{\N\setminus (A\triangle B)}.
$$
But we have
\begin{align*}
2\phi_A \phi_B-\phi_A-\phi_B+1 &= 2\phi_A \phi_B-\phi_A+f_B\theta_B\\
&= \frac{\theta}{\theta_A}g_A(2\phi_B-1)+f_B\theta_B
\end{align*}
which is clearly divisible by $\theta_{B\setminus A}$. By symmetry, we also find that it is divisible by $\theta_{A\setminus B}$. Since these last two inner functions do not have a common inner factor, we conclude that
$$
2\phi_A \phi_B-\phi_A-\phi_B+1
$$
is divisible by $\theta_{(A\setminus B)\cup (B\setminus A)}=\theta_{A\triangle B}.$
Thus,
$$
2\phi_A \phi_B-\phi_A-\phi_B+1-\phi_{\N\setminus (A\triangle B)}.
$$
is divisible by $\theta_{A\triangle B}$.
On the other hand, using (\ref{coronaAB}) we can write
\begin{align*}
2\phi_A \phi_B-\phi_A-\phi_B+1&= \phi_A(\phi_B-1)+\phi_B(\phi_A-1)+1\\
&=-\phi_A f_B\theta_B-\phi_Bf_A\theta_A+1.
\end{align*}
Therefore, another application of (\ref{coronaAB}) yields
\begin{align*}
2\phi_A \phi_B-\phi_A-\phi_B+1-\phi_{\N\setminus (A\triangle B)}&=1-f_A\theta_A\phi_B-f_B\theta_B\phi_A\\
&+(f_{\N \setminus (A\triangle B)}\theta_{\N \setminus (A\triangle B)}-1)\\
&=-f_Ag_B\theta_A\frac{\theta}{\theta_B}-f_Bg_A\theta_B\frac{\theta}{\theta_A}\\
&+f_{\N \setminus (A\triangle B)}\theta_{\N \setminus (A\triangle B)}
\end{align*}
which is easily checked to be divisible by $\theta_{\N \setminus (A\triangle B)}$. Coupled with the previously established divisibility relation, we conclude that $\theta$ divides
$$
2\phi_A \phi_B-\phi_A-\phi_B+1-\phi_{\N\setminus (A\triangle B)}
$$
which in turn implies that $k_A k_B=k_{\N\setminus (A\triangle B) }$.
Note now that if $A=\N$, then we can choose $f_A=0,g_A=1$ and we get $I=k_{\N}$. To verify that $G$ is an abelian multiplicative subgroup of $B(\hil)$, it therefore only remains to check that $A\triangle B$ is finite or cofinite whenever $A$ and $B$ are, but this is elementary.
\end{proof}

Recall now a classical result of Dixmier (see \cite{dixmier} and \cite{paulsen}).
\begin{theorem}\label{unitdix}
Let $G$  be an abelian multiplicative subgroup of $B(\hil)$. Assume that 
$$
C=\sup \{\|k\|_{B(\hil)}:k\in G\}<\infty.
$$
Then, there exists an invertible operator $X\in
B(\hil)$ such that $XkX^{-1}$ is unitary for every $k\in G$ and with the property that 
$$
\max\{\|X\|,\|X^{-1}\|\}\leq C.
$$

%
\end{theorem}

Combining Lemma \ref{group} and Theorem \ref{unitdix}, we obtain (see Proposition 3.2 of \cite{clouatre} for more details) the following theorem which was proved in \cite{Vas2} by different means (see also \cite{NikVas} for an English translation).

\begin{theorem}\label{directsum}
Let $T\in B(\hil)$ be an operator of class $C_0$. Assume that the minimal function of $T$ is $\theta=\prod_{n}\theta_n$ where $\{\theta_n\}_n$ satisfies the generalized Carleson condition with constant $C>0$. Then, there exists an invertible operator $Y$ such that
$$
YTY^{-1}=\bigoplus_{n}T|\ker \theta_n(T)
$$
and with the property that $\max\{\|Y\|, \|Y^{-1}\|\}\leq (2C+1)^2$.
\end{theorem}

\section{A similarity result}
In this section we apply Theorem \ref{t-simalg} and Theorem \ref{directsum} to similarity questions for multiplicity-free operators of class $C_0$. Let us describe the precise setting for the result we wish to prove. 
We assume that the minimal function is a Blaschke product $\theta$ which can be written as $\theta=\prod_n \theta_n$, where $\{\theta_n\}_n\subset H^\infty$ is a sequence of finite Blaschke products with at most $N$ roots satisfying the generalized Carleson condition. 

A particular case of this situation is that where $\theta_n=b_{\lambda_n}^N$ and $\{\lambda_n\}_n \subset \D$ is a sequence satisfying the Carleson condition (see \cite{nikOP} for details). This is the case covered by the main result of \cite{clouatre}, stated below.

\begin{theorem}\label{similar2}
Let $\{\lambda_n\}_n\subset \D$ be a
sequence satisfying the Carleson condition and $\{m_n\}_n\subset \N$ be a bounded sequence. Let $T\in B(\hil)$ be
a multiplicity-free operator of class $C_0$ with minimal function 
$$
\theta(z)=\prod_{n}\left(\frac{\ol{\lambda_n}}{|\lambda_n|}\frac{z-\lambda_n}{1-\ol{\lambda_n}z} \right)^{m_n}.
$$ 
Assume that $\phi(T)$ has closed range for every inner divisor $\phi$ of $\theta$. Then, $T$ is similar to $S(\theta)$.
\end{theorem}

We would now like to tackle the case where the functions $\theta_n$ might have distinct roots.  Such functions have been studied in the context of interpolation (see \cite{hart1}, \cite{hart2},\cite{NikShift} and \cite{Vas2}). 
Before proceeding, we make an addition to Theorem \ref{equivalence2}. We hope that it can shed some light on the condition that will appear in our similarity result, especially with regard to the assumption of Theorem \ref{similar2}.

\begin{theorem}\label{t-equivalencecarleson}
Let $\{\lambda_n\}_n\subset \D$ be a sequence of distinct points and let $\theta \in H^\infty$ be the corresponding Blaschke product. The following statements are
equivalent:
\begin{enumerate}
\item[\rm{(i)}] $\{\lambda_n\}_n\subset \D$ satisfies the Carleson condition
\item[\rm{(ii)}] every multiplicity-free operator $T$ of class $C_0$ with minimal function $\theta$ is similar to $S(\theta)$
\item[\rm{(iii)}] $\phi(T)$ has closed range for every multiplicity-free operator $T$ of class $C_0$ with minimal function $\theta$ and every inner divisor $\phi$ of $\theta$
\item[\rm{(iv)}] for every multiplicity-free operator $T$ of class $C_0$ with minimal function $\theta$, there exists a constant $\beta>0$ such that
$$
\|u(T)|\ker \phi(T)\|\geq \beta \|u\|_{H^\infty/\phi H^\infty}
$$
for every $u\in H^\infty$ and every inner divisor $\phi$ of $\theta$.
\end{enumerate}
\end{theorem}
\begin{proof}
The equivalence of (i), (ii) and (iii) is from \cite{clouatre}. To see that (ii) implies (iv), assume $XTX^{-1}=S(\theta)$. By Lemma \ref{l-unitequiv}, we know that $u(S(\theta))|\ker \phi(S(\theta))$ is unitarily equivalent to $u(S(\phi))$ for every $u\in H^\infty$ and every inner divisor $\phi$ of $\theta$, so that
$$
\|u(S(\theta))|\ker \phi(S(\theta))\|=\|u(S(\phi))\|=\|u\|_{H^\infty/\phi H^\infty}.
$$ 
Note moreover that if we set $Y_{\phi}=X|\ker \phi(T):\ker \phi(T)\to \ker \phi(S(\theta))$, then
$$
Y_{\phi}u(T)|\ker \phi(T)Y_{\phi}^{-1}=u(S(\theta))|\ker \phi(S(\theta))
$$
and therefore
$$
\|u(T)|\ker \phi(T)\|\geq \frac{1}{\|Y_{\phi}\| \|Y_{\phi}^{-1}\|}\|u\|_{H^\infty/\phi H^\infty}\geq \frac{1}{\|X\| \|X^{-1}\|}\|u\|_{H^\infty/\phi H^\infty}
$$
which is (iv).
To finish the proof, it suffices to show that (iv) implies (i). Let $\psi_n=\theta/b_{\lambda_n}$. By definition, we need to show that
$|\psi_n(\lambda_n)|\geq \beta$ for some $\beta>0$. Arguing as in the proof of Corollary \ref{c-simalg}, we see that
$$
\|\psi_n\|_{H^\infty/\theta H^\infty}=1.
$$
If we let 
$$
T=\bigoplus_{n}S(b_{\lambda_n})=\bigoplus_{n}\lambda_n
$$
then
$$
\psi_n(T)=0\oplus 0\oplus \ldots \oplus \psi_n(\lambda_n)\oplus 0 \oplus \ldots
$$
Using (iv), we find
$$
|\psi_n(\lambda_n)|=\|\psi_n(T)\|\geq \beta \|\psi_n\|_{H^\infty/\theta H^\infty} \geq \beta
$$
and we are done.
\end{proof}

We can now formulate our similarity result.

\begin{theorem}\label{t-similargc}
Let $\theta\in H^\infty$ be a Blaschke product which can be written as $\theta=\prod_n \theta_n$, where $\{\theta_n\}_n\subset H^\infty$ is a sequence of finite Blaschke products with at most $N$ roots satisfying the generalized Carleson condition. Define
$$
\eta=\sup_{n}\sup_{\lambda,\mu\in \theta_n^{-1}(0)}\frac{|b_{\lambda}(\mu)|^{1/2}}{(1-\max \{|\lambda|,|\mu| \}^2)^{1/2}}.
$$
Let $T\in B(\hil)$ be a multiplicity-free operator of class $C_0$ with minimal function $\theta$. Assume that there exists a constant $\beta$ such that
$$
\sqrt{1-\frac{1}{(N-1)^2}}< \beta-5\sqrt{2}\eta<1
$$
and
$$
\|u(T)|\ker \theta_n(T)\|\geq \beta \|u\|_{H^\infty/\theta_n H^\infty}
$$
for every $u\in H^\infty$ and every $n$.
Then, $T$ is similar to $S(\theta)$.
\end{theorem}
\begin{proof}
By Theorem \ref{directsum}, we have that $T$ is similar to 
$$
\bigoplus_{n}T|\ker \theta_n(T)
$$
and that $S(\theta)$ is similar to 
$$
\bigoplus_{n}S(\theta)|\ker \theta_n(T).
$$
Using Lemma \ref{l-unitequiv}, we see that $S(\theta)$ is actually similar to $\bigoplus_{n}S(\theta_n)$. Therefore, it is sufficient to show that for each $n$ the operator $T|\ker \theta_n(T)$ is similar to $S(\theta_n)$ via an invertible operator 
$$
X_n:\ker \theta_n(T)\to H(\theta_n)
$$
satisfying 
$$
\sup_{n} \{\|X_n\|,\|X_n^{-1}\| \}<\infty.
$$
But this follows from Corollary \ref{c-simalg}. Indeed, for each $n$ we can define 
$$
\psi_n=\theta_n/b_{\lambda_n}
$$
for some $\lambda_n\in \theta_n^{-1}(0)$. Then by assumption we have
$$
\|\psi_n(T)|\ker \theta_n(T)\|\geq \beta\|\psi_n\|_{H^\infty/\theta_n H^\infty}
$$
and as in the proof of Corollary \ref{c-simalg}, we see that
$$
\|\psi_n\|_{H^\infty/\theta H^\infty}=1.
$$
When combined, these inequalities yield
$$
\|\psi_n(T)|\ker \theta_n(T)\|\geq \beta.
$$
In view of our assumption on $\eta(\theta)$, we can apply Corollary \ref{c-simalg} to finish the proof.
\end{proof}
The reader should note at this point that the inequalities appearing in the previous theorem implicitely force $\eta\leq(5\sqrt{2})^{-1}$ since obviously $0<\beta\leq 1$. 

If we restrict our attention to the special case where each $\theta_n$ has two roots, then we obtain a simpler statement.

\begin{corollary}\label{c-similargc}
Let $\theta\in H^\infty$ be a Blaschke product which can be written as $\theta=\prod_n \theta_n$, where $\{\theta_n\}_n\subset H^\infty$ is a sequence of finite Blaschke products with at most two roots satisfying the generalized Carleson condition. Let $T\in B(\hil)$ be a multiplicity-free operator of class $C_0$ with minimal function $\theta$. 
Assume that there exists a constant $\beta$ such that
$$
\|u(T)|\ker \theta_n(T)\|\geq \beta \|u\|_{H^\infty/\theta_n H^\infty}
$$
for every $u\in H^\infty$ and every $n$. Then, $T$ is similar to $S(\theta)$.
\end{corollary}
\begin{proof}
Proceed as in the proof of Theorem \ref{t-similargc}, but use Corollary \ref{c-dim2} instead of Corollary \ref{c-simalg}.
\end{proof}

Let us close by making a few comments about $\eta$. For that purpose, let us introduce another quantity related to $\theta$,
$$
\delta=\inf_{n}\inf_{\lambda,\mu\in \theta_n^{-1}(0)}|b_{\lambda}(\mu)|^{1/2}.
$$
If $\delta>0$, then a similarity result analogous to Theorem \ref{t-similargc} follows immediately from Theorem \ref{directsum} (as was pointed out in the proof of Corollary \ref{c-dim2}). On the other hand, if $\eta=0$ then we are in the case covered by Theorem \ref{similar2}. Now, it is obvious that
$
0\leq \delta\leq \eta.
$
This shows that our (implicit) condition $\eta\leq(5\sqrt{2})^{-1}$ closes part of the gap between those two cases where similarity was known previously.

\end{document}